\documentclass[10pt]{amsart}
\usepackage{amsmath,amssymb}
\newtheorem{bigthm}{Theorem}
\newtheorem{thm}{Theorem}[section]

\newtheorem{lem}{Lemma}[section]
\newtheorem{defn}{Definition}[section]

\numberwithin{equation}{section}

\title{Pointwise Convergence of Ergodic Averages in Orlicz Spaces}
\author{Andrew Parrish}

\begin{document} 

\maketitle
\begin{abstract}
	We construct a sequence ${a_n}$ such that for any aperiodic measure-preserving system $(X, \Sigma, m, T)$ the ergodic averages 
	\begin{equation*}
	 A_Nf(x) = \frac{1}{N} \sum_{n=1}^N f(T^{a_n}x)
	\end{equation*}
	converge a.e. for all $f$ in $L \log \log (L)$ but fail to have a finite limit for an $f \in L^1$.  In fact, we show that for each Orlicz space properly contained in $L^1$ there is a sequence along which the ergodic averages converge for functions in the Orlicz space, but diverge for all $f \in L^1$. This extends the work of K. Reinhold, who, building on the work of A. Bellow, constructed a sequence for which the averages $A_N f(x)$ converge a.e. for every $f \in L^p$, $p >q \geq 1$, but diverge for some $f \in L^q$.  Our method, introduced by Bellow and extended by Reinhold and M. Wierdl, is perturbation.
\end{abstract}

\tableofcontents

\section{Introduction and Preliminaries}
In this paper, we give a method for constructing sequences along which ergodic averages converge a.e. for functions in a certain Orlicz space $L \phi(L)$, yet diverge for a function in $L^1$.  All measure-preserving systems mentioned should be understood to be aperiodic (free) and of finite (probability) measure. 

If $A$ is a set of integers, $|A|$ denotes the cardinality of $A$ and $A(N) = A \cap [1, N)$.  $\{a_n\}=S$ will always denote an increasing sequence of positive integers.

\begin{defn}
 If $F$ is a function space, then $S$ is \textit{universally $F$-good} iff the sequence of averages 
 	 \begin{equation*}
		A_N[S,f](x)= A_Nf(x) = \frac{1}{N} \sum_{n=1}^N f(T^{a_n}x)
	 \end{equation*}
converges a.e. for every $f \in F$ for all measure-preserving systems.

$S$ is \textit{universally $F$-bad}  if there is an $f \in F$ for which the limit fails to exist for all $x$ in a set of positive measure for all measure-preserving systems. 
\end{defn}

\begin{defn}
We say that $S$ is \textit{universally $\infty$-sweeping out in $F$ } if for all measure-preserving systems there exists an $f \in F$ such that
	\begin{equation*}
		 \sup_N \frac{1}{N} \sum_{n=1}^N f(T^{a_n}x) = \infty.
	 \end{equation*}
\end{defn}

The Pointwise Ergodic Theorem shows that the natural numbers are universally $L^1$-good.  The existence of sequences of zero density that are universally $L^1$-good was proven in \cite{Bellow1984}.  In \cite{Bourgain1989Arithmetic}, it was shown that the sequence of squares is universally $L^p$-good, for $p>1$.  In \cite{Wierdl1988}, we see that the sequence of primes are universally $L^p$-good for $p>1$, as well. Recently it has been shown that the squares (\cite{Buczolich2008squares}) and primes (\cite{LaVictoire2009}) are $L^1$-bad. More may be found in \cite{Boshernitzan1996}: for example, if $[x]$ is the largest integer less than or equal to $x$,  $\{ [n^2 \log\log n] \}$  is $L^p$-good for $p>1$, and $\{[\sqrt{n} \log (\sqrt{n})]\}$ is $L^1$-good. 

In \cite{Bellow1989Perturb}, Bellow constructs a universally $L^p$-good sequence that is universally $L^q$-bad for $1 \leq q < p$ and any $1 < p < \infty$.  Using similar methods, Reinhold \cite{Reinhold1994} showed that there is a sequence which is universally $L^p$-good for $p > q \geq 1$ but universally $L^q$-bad and constructed sequences which are $L^q$-bad for all $q < \infty$ but good in $L^\infty$.  Our proof will use same method, perturbation, but will build on a different approach, explored in \cite{Wierdl1998Perturb}.

\begin{defn}
 Let $S$ be a strictly increasing sequence of positive integers.  A sequence $\Delta$ is a perturbation of $S$ iff
 \begin{equation*}
  \lim_{N \rightarrow \infty} \frac{|\{(\Delta \backslash S) \cup (S \backslash \Delta) \}(N)|}{|S(N)|} = 0.
 \end{equation*}
\end{defn}

Suppose $\phi: \mathbb{R} \rightarrow \mathbb{R}^+$ is strictly increasing, unbounded, and $\phi(x) = 1$ when $x \leq 1$.  For a probability space $(X, \Sigma, m)$, we define
\begin{equation*}
 L \phi (L) = \left\{ f: X \rightarrow \mathbb{R} \, \left|   \, \int_X |f(x)| \, \phi(|f(x)|) \, dx < \infty \right. \right\}.
\end{equation*}

Notice that $|x| \phi(|x|)$ is a Young's function; $L \phi(L)$ is therefore an Orlicz space.  Conversely, if $X$ is a finite measure space, it can be shown that for every Orlicz space $L^\Phi (X)$ containing $L^p$ for all $p>1$ and contained in $L^1(X)$, there is a $\phi$ so that $L^\Phi = L \phi(L)$ \cite{RaoRen}. Similarly, if $1 < q <\infty$, an Orlicz space containing $L^q$ and contained in $L^1$ is equivalent to 
\begin{equation*}
 \frac{L^q}{\phi \left( L \right)} = \left\{f : \int_X \frac{|f|^q}{\phi \left( |f|\right)} \, dm < \infty \right\}
\end{equation*}
for some strictly increasing, unbounded $ \phi : \mathbb{R} \rightarrow \mathbb{R}^+$ with $\phi(x)=1$ for $x \leq 1$, so long as $\frac{|x|^{q-1}}{\phi(x)} \rightarrow \infty$ as $x \rightarrow \infty$.  

We will use this notation, with $L \phi(L)$ or $\frac{L^q}{\phi \left( L \right)}$ representing our Orlicz space rather than the traditional $L^\Phi (X)$, for two reasons: first, because the variation on Yano's Extrapolation Theorem in Section \ref{YanoSec} will be stated in terms of $\phi$, and second, because $\phi$ plays an important role in Lemma \ref{notasgood}. Using this notation, our results may be written as follows.

\begin{bigthm}\label{prebig}
 Suppose $ \phi : \mathbb{R} \rightarrow \mathbb{R}^+$ is strictly increasing, unbounded, $\phi(x)=1$ for $x \leq 1$, and $\phi(x) << |x|^q$.  Let $S$ be a zero-density sequence that is universally $\frac{L^q}{\phi \left( L \right)}$-good, $1 < q < \infty$. Then there exists a perturbation of $S$ that is universally $L^q$-good but universally $\infty$-sweeping out in $\frac{L^q}{\phi \left( L \right)}$.
\end{bigthm}

\begin{bigthm}\label{big}
 Let $S$ be a zero-density sequence that is universally $L^1$-good. Then for any strictly increasing, unbounded $\phi: \mathbb{R} \rightarrow \mathbb{R}^+$, with $\phi(x)=1$ for $x \leq 1$, there exists a perturbation of $S$ that is universally $L \phi (L)$-good but universally $\infty$-sweeping out in $L^1$.
\end{bigthm}

\begin{lem}\label{notasgood}
 Suppose $\psi$ and $\phi$ are each strictly increasing, unbounded real-valued functions with $\psi(x) = \phi(x) = 1$ for $x \leq 1$,  and $\phi(x) << \log^j(x)$ for all $j > 0$.  If there exists a $k$ so that $\frac{[\phi(u)]^\frac{k}{k+1}}{\psi(u)} \rightarrow \infty$ then there is a sequence that is universally $L \phi(L)$-good but universally $\infty$-sweeping out for $L \psi(L)$
\end{lem}

Our strategy in Theorem \ref{prebig} will be to adapt the approach of \cite{Wierdl1998Perturb} to the Orlicz space setting.  Using a similar adaptation and a slight generalization of Yano's extrapolation theorem from \cite{Yano1951}, we prove Theorem \ref{big}.  We will then show how these methods can lead us to Lemma \ref{notasgood}.

\section{Yano's Extrapolation Theorem}\label{YanoSec}
Theorem \ref{Yanos} is a slight generalization of Shigeki Yano's extrapolation result from \cite{Yano1951}, and the proof follows Yano's original proof closely.  The main idea of the proof is first to take advantage of the sublinearity of both the operator and the norm to disassemble the function, apply the assumed inequality, and finally to reassemble the function from the pieces.  

\begin{thm}\label{Yanos}
Suppose $T$ is a positive sublinear operator taking measureable functions to measureable functions and bounded on $L^\infty$,  $X$ is a probability space with measure $m$, and $\phi: \mathbb{R^+} \rightarrow \mathbb{R}$ is a nondecreasing function so that
\begin{enumerate}
 \item[(a)] $\phi(x)= 1$ if $x \leq 1$, and
 \item[(b)] $\phi(x) \leq c\sqrt{x}$ for some constant $c$ and all $x>1$.
\end{enumerate}

Let $\varepsilon > 0$.  If, for a measureable function $f: X \rightarrow \mathbb{R}$, we have
 \begin{equation} \label{assum}
  \left\|  Tf \right\|_{L^p} \leq  \phi \left( e^\frac{1}{p-1} \right)  \left\| f \right\|_{L^p} 
 \end{equation}
for all $p$, $1 < p \leq 1 + \varepsilon$, then there is a constant $A_{\phi}$ so that
\begin{equation}
  \left\|  Tf \right\|_{L^1} \leq A_{\phi} + 4e^3 \int_X |f| \, \phi ( |f| ) \, dm.
 \end{equation}

\end{thm}

\begin{proof}

 Let $f: [0,1] \rightarrow \mathbb{R}$ be a measureable function. Then $f= f^+ - f^-$, where $f^+$ and $f^-$ are the positive and negative parts of $f$.

Define $g^+= \frac{f^+}{2} + 1$ and $g^-= \frac{f^-}{2} + 1$, and let $E_n = \{x : e^{n} \leq g^+ < e^{(n+1)}\}$.

Then $\sum_{n=0}^\infty e^{n} \chi_{E_n} \leq g^+ \leq \sum_{n=0}^\infty e^{n+1} \chi_{E_n}$.

By sublinearity and positivity of $T$ we have 
\begin{equation*}
 |Tg^+| \leq \sum_{n=0}^\infty e^{(n+1)} |T \chi_{E_n}|.
\end{equation*}
Integrating and then applying H\"{o}lder's Inequality, we have 
\begin{align*}
 \int_X |Tg^+| \, d\mu  &\leq \sum_{n=0}^\infty e^{(n+1)} \int_X |T \chi_{E_n}|\, dm \\
			&\leq \sum_{n=0}^\infty e^{(n+1)} \left\| T \chi_{E_n} \right\|_{L^{p_n}} \cdot m(X). 
\end{align*}

Let $p_n = 1 + \frac{1}{n}$. Then, applying the assumption (\ref{assum}), 
\begin{equation}
  \int_X |Tg^+| \, dm \leq  \sum_{n=0}^\infty e^{(n+1)} \phi \left(e^n\right) \left(m\left( E_n \right)\right)^\frac{n}{n+1}. \label{EnSum}
\end{equation} 

If $m\left( E_n \right) < e^{-2(n+1)}$ then the sum in (\ref{EnSum}) converges to a constant $a_{\phi}$.  

If $m\left( E_n \right) \geq e^{-2(n+1)}$, we have 
\begin{align*}
\sum_{m\left( E_n \right) \geq e^{-2(n+1)}} e^{(n+1)} &\phi \left(e^n\right) \left(m\left( E_n \right)\right)^\frac{n}{n+1} \\
&= \sum_{m\left( E_n \right) \geq e^{-2(n+1)}} e^{(n+1)} \phi \left(e^n\right) m\left( E_n \right)  \left(m\left( E_n \right)\right)^\frac{-1}{n+1} \\
& \leq \sum_{m\left( E_n \right) \geq e^{-2(n+1)}} e^{(n+1)} \phi \left(e^n\right) m\left( E_n \right)  \left( e^{-2(n+1)} \right)^\frac{-1}{n+1} \\
& = e^3 \sum_{ m\left( E_n \right) \geq e^{-2(n+1)}} e^{(n+1)} \phi \left(e^n\right) m\left( E_n \right) \\
&\leq e^3 \int_X g^+ \, \phi \left(g^+ \right) \, dm.
\end{align*}

Hence, 
\begin{align*}
  \int_X |Tg^+| \, dm &\leq a_{\phi} + e^3  \int_X g^+ \, \phi \left(g^+\right)  \, dm.
\end{align*}

Similarly, $\int_X |Tg^-| \, dm \leq  {a'}_{\phi} + e^3  \int_X g^-  \, \phi \left(g^-\right)  \, dm.$
\newline

By the sublinearity of T, we have
\begin{align*}
\int_X \left|T\left(\frac{f}{2}\right) \right| \, dm &\leq \int_X |Tg^+| \, dm + \int_X |Tg^-| \, dm \\
&\leq  a_{\phi} + e^3  \int_X \left(g^+ \right)  \phi \left(g^+\right)  \, dm + {a'}_{\phi} + e^3  \int_X \left(g^- \right)  \phi \left(g^-\right)  \, dm
\\
&\leq  (a_{\phi}+ {a'}_{\phi}) + 2e^3 \int_X \left( \frac{|f|}{2} + 1\right) \, \phi \left( \frac{|f|}{2} + 1\right)  \, dm \\ 
\end{align*}

But 
\begin{align*}
 &2e^3 \int_X \left( \frac{|f|}{2} + 1\right) \, \phi \left( \frac{|f|}{2} + 1\right)  \, dm \\
&=  2e^3 \left\{ \int_{\{\frac{|f|}{2} > 1\}} \left( \frac{|f|}{2} + 1\right) \, \phi \left( \frac{|f|}{2} + 1\right)  \, dm + \int_{\{\frac{|f|}{2} \leq 1\}} \left( \frac{|f|}{2} + 1\right) \, \phi \left( \frac{|f|}{2} + 1\right) \, dm \right\}\\
&\leq 2e^3  \left\{ \int_{\{\frac{|f|}{2} > 1\}}  |f|  \, \phi \left( |f|  \right)  \, dm + \int_{\{\frac{|f|}{2} \leq 1\}} 2   \, \phi( 2 )  \, dm \right\}\\
&< 2e^3 \int_X  |f|  \, \phi\left( |f| \right)  \, dm + 4e^3 \phi(2).
\end{align*}

Applying the sublinearity of $T$ once more, 
\begin{equation*}
  \frac{1}{2} \int_X \left|T\left( f \right) \right| \, dm \leq a_{\phi}+ {a'}_{\phi} + 4e^3 \phi(2) + 2e^3 \int_X  |f|  \, \phi\left( |f| \right)  \, dm.
\end{equation*}

Letting $A_\phi = 2a_{\phi}+ 2{a'}_{\phi} + 8e^3 \phi(2)$ we are done.
\end{proof}

\section{Proof of Theorem \ref{prebig}}
To construct our sequence, we will begin with any zero-density, universally $\frac{L^q}{\phi \left( L \right)}$-good sequence.  The zero-density property gives us large gaps in the sequence, into which we will insert sets of ``badly behaved'' elements.  The number of these added elements will be very small relative to the number of elements of our original sequence up to the point of their inclusion, thereby guaranteeing that our new sequence will be a perturbation of the original.  They will be of sufficient number, however, to insure the failure of the relevant maximal inequality.  We will then seek a bound on the $L^q$ norm of the maximal operator.

In order to show that our constructed sequence is universally $\infty$-sweeping out for $\frac{L^q}{\phi \left( L \right)}$, we will make use of the following lemma (and associated definition) adapted from \cite{Wierdl1998Perturb}.

\begin{defn}
 Suppose $f: \mathbb{Z} \rightarrow \mathbb{R}$. Define 
 \begin{equation*}
  D(f) = \limsup_{N \rightarrow \infty} \frac{1}{2N+1} \sum_{n = -N}^N |f(n)|.
 \end{equation*}
Further, if $A$ is a measureable set, we define $D(A) = D(\chi_A)$.
\end{defn}

\begin{lem}\label{badseq}
 Let $\Phi(x) : \mathbb{R} \rightarrow \mathbb{R}$ be a non-decreasing function with $ 1 << \Phi(x) << x^r$ for some real constant $r>0$, and denote by $\Phi(L)$ the set of real-valued functions $\left\{f : \int_X \Phi \left(|f|\right) \, dm < \infty \right\}$. Let $S$ be a strictly increasing sequence of positive integers.  If for every positive $K$ and $\varepsilon$ there is an $f: \mathbb{Z} \rightarrow \mathbb{R}$, with $D\left( \Phi\left( f \right) \right) \leq 1$, and a finite set of integers $\Lambda$ so that
 \begin{equation*} 
  D \left\{ n: \max_{N \in \Lambda} \frac{1}{|S(N)|} \sum_{m \in S(N)} f(n+m) \geq K \right\}  \geq 1 - \varepsilon
 \end{equation*}
then S is a universally $\infty$-sweeping out sequence for $\Phi(L)$.
\end{lem}
The proof does not significantly differ from that presented in \cite{Wierdl1998Perturb}.

\begin{proof}[Proof of Theorem \ref{prebig}]
 Suppose $M(u): \mathbb{N} \rightarrow \mathbb{N} \cup \{ 0 \} $ is a nondecreasing function. For $u= 1, 2, 3, ...$ define a sequence of sets $\left\{ A_u \right\}$ as follows.
\begin{align*}
 A_1 &= \left\{ 0, 1, 2, ..., M(1)-1 \right\}\\
 A_2 &= \left\{ M(1), M(1)+1, ..., M(1)+ M(2) - 1\right\}\\
 \vdots\\
 A_u &= \left\{ \sum_{j=1}^{u-1} M(j), \sum_{j=1}^{u-1} M(j) + 1, ..., \sum_{j=1}^{u} M(j) -1 \right\}\\
 \vdots
\end{align*}
So every positive integer is contained in some $A_u$, $A_i$ and $A_j$ are disjoint for all $i \neq j$, and $|A_u| = M(u)$ for all $u$. \\

Let $\{n_k\}$ be a sequence with properties to be discussed below.  To create our perturbation, $\Delta$, we will add a certain number of elements to $S$ from each interval $[n_k, 2n_k)$. Let $R(x): \mathbb{R} \rightarrow \mathbb{R}^+$ be a decreasing function with $\lim_{x \rightarrow \infty} R(x)= 0$. From each interval $[n_k, 2n_k)$, we will add $R(u) |S(n_k)|$ elements congruent to $k \mod M(u)$ if $k \in A_u$.\\

We will need each $[n_k, 2n_k)$ to be disjoint from the next.  Hence we require that $n_k > 2n_{k-1}$.  We also will need each interval to be large enough to contain our added elements.  Since there is an integer congruent to $k \mod M(u)$ in every $M(u)$ consecutive integers, we require the length of each interval, $n_k$, to be greater than
\begin{equation*}
 R(u)|S(n_k)| \cdot M(u), \, k \in A_u.
\end{equation*}

We want the number of elements of $S$ in the interval $[n_k, 2n_k)$ to be small relative to the length of the interval.  This will help us insure that the added elements upset the relevant maximal inequality.  We require $|S(2n_k)| \leq 3 |S(n_k)|$.  Finally, choosing our intervals so that 
\begin{equation*}
 R(u)|S(n_k)| > \sum_{j=1}^{k-1} |S(n_j)|
\end{equation*}
will help insure that $\Delta$ is, in fact, a perturbation.\\

We can satisfy both our disjointness and perturbation requirements by choosing $n_k$ large enough. To see that we may likewise choose $n_k$ in such a way as to satisfy the other two, consider that since our original sequence has density zero, there must be a sequence of positive integers $\{ m_j \}$ such that 
\begin{equation*}
 \frac{|S(m_j)|}{m_j} \leq \frac{|S(m)|}{m} \mbox{, for } m \leq m_j.
\end{equation*}
Let $n_k = [m_j/2]$.  We then have 
\begin{equation*}
 \frac{|S(n_k)|}{n_k} \leq \frac{|S([m_j/2])|}{[m_j/2]} \leq 3 \frac{|S(m_j)|}{m_j}.
\end{equation*}
Since $\frac{|S(m_j)|}{m_j} \rightarrow 0$ monotonically as $j \rightarrow \infty$, we can choose $j$ large enough so that 
\begin{equation*}
 \frac{|S(n_k)|}{n_k} \leq \frac{1}{R(u)M(u)}.
\end{equation*}
Further, we have
\begin{align*}
 |S(2n_k)| \leq |S(m_j)| &= \frac{|S(m_j)|}{m_j} \cdot m_j \leq \frac{|S([m_j/2])|}{[m_j/2]} \cdot m_j \\
& \leq 3 |S([m_j/2])| = 3|S(n_k)|,
\end{align*}
our third requirement.\\

Having constructed $\Delta$, we will now show that it is a perturbation.
Since $\Delta$, is formed by adding new terms to $S$, we need only show that
\begin{equation*}
 \lim_{n \rightarrow \infty} \frac{|\Delta (n)  \setminus S|}{|S(n)|} = 0.
\end{equation*}

For any $n$ sufficiently large, there is a $k$ and $u$, with $k \in A_u$ so that $n_k \leq n < n_{k+1}$.
Then, using that $R(u)|S(n_k)| > \sum_{j=1}^{k-1} |S(n_j)|$ and $R(u)^2 < R(u)$, we have
\begin{equation} \label{allless}
 |\Delta (n) \setminus S| \leq R(u) \left( |S(n_k)| + \sum_{j=1}^{k-1} |S(n_j)| \right) < 2 R(u) |S(n_k)|.
\end{equation}

Since $|S(n)| \geq |S(n_k)|$, 
\begin{equation*}
 \frac{|\Delta (n) \setminus S(n)|}{|S(n)|} \leq  2 R(u) |S(n)| \cdot \frac{1}{|S(n)|} = 2 R(u).
\end{equation*}
This goes to 0 as $n \rightarrow \infty$ since as $n$ goes to infinity so do $k$ and $u$.\\

To complete the construction of our perturbation, we will now let 
\begin{align*}
 M(u)&= \left\lfloor \frac{\left\{ \phi^{-1}(u^3) \right\}^q}{u^3} \right\rfloor  \mbox{ and}\\ 
 R(u)&= \frac{u^\frac{1}{q}}{\phi^{-1} (u^3)}
\end{align*}

If $n$ is an integer then for every $u$ there must be some $k \in A_u$ so that $n \equiv -k \mod M(u)$.  By our construction of $\Delta$, there are at least $R(u)|S(n_k)|$ integers congruent to $k \mod M(u)$ in $\Delta \cap [n_k, 2n_k)$.  Let 
\begin{equation*}
 E_k= \left\{ m \in \Delta \cap [n_k, 2n_k) : m \equiv k \mod M(u) \right\} 
\end{equation*}

Define $F: \mathbb{Z} \rightarrow \mathbb{R}$ by
\begin{equation*}
 F(n) = \left\{ \begin{array}{l}  \phi^{-1}(u^3)  , \quad  \mbox{if } n= 0 \mbox{ mod } M(u) \\ 0 , \quad  \mbox{ otherwise }  \end{array} \right.
\end{equation*}

Note that $F(n+m) = \phi^{-1}(u^3)$ for all $m \in E_k$.  If we let $\Phi(x) = \frac{|x|^q}{\phi(|x|)}$, we also have
\begin{align*}
 D\left( \Phi\left( F \right) \right) &=   \limsup_{N \rightarrow \infty} \frac{1}{2N+1} \sum_{n = -N}^N \Phi \left( F(n) \right)\\
&=  \limsup_{N \rightarrow \infty} \frac{1}{2N+1} \sum_{\substack{-N \leq n \leq N \\ n \equiv 0 \mbox{ mod } M(u) }} \frac{\left(\phi^{-1}(u^3)\right)^q}{u^3}\\ 
&=  \limsup_{N \rightarrow \infty}   \frac{\left| \left\{-N \leq n \leq N : \, n \equiv 0 \mbox{ mod } M(u) \right\} \right|}{2N+1} \cdot \frac{\left(\phi^{-1}(u^3)\right)^q}{u^3}  \\
&= \frac{1}{M(u)} \cdot \frac{\left(\phi^{-1}(u^3)\right)^q}{u^3} \\
&\leq 1.
\end{align*}
 We will now apply Lemma \ref{badseq}.

Fix $u$ large enough so that 
\begin{equation*}
 |\Delta (2n_k)| = |S(2n_k)| +  |\Delta(2n_k) \setminus S(2n_k)|  \leq 3|S(n_k)|  + |\Delta(2n_k) \setminus S(2n_k)| \leq 4|S(n_k)|.
\end{equation*}
We may then estimate 
\begin{align*}
 \max_{k \in A_u} \frac{1}{|\Delta(2N_k)|} \sum_{m \in \Delta(2N_k)} F(n+m) &\geq  \frac{1}{4|S(N_k)|} \sum_{m \in E_k} F(n+m) \\
&= \frac{1}{4|S(N_k)|} \sum_{m \in E_k} \phi^{-1}(u^3).
\end{align*}
By our construction of $\Delta$, however, there must be at least $ R(u)|S(N_k)|$ elements in $E_k$. So
\begin{align*}
 \frac{1}{4|S(N_k)|} \sum_{m \in E_k} \phi^{-1}(u^3)
&\geq \frac{1}{4|S(N_k)|} \cdot R(u)|S(N_k)| \cdot \phi^{-1}(u^3)\\
&= \frac{u^\frac{1}{q}}{4}.
\end{align*} 
By Lemma \ref{badseq}, then, we have that $\Delta$ is universally $\infty$-sweeping out in $\frac{L^q}{\phi \left( L \right)}$.

We will now show that $\Delta$ is universally $L^q$-good.
Let $(X,\mathcal{B}, \mu, T)$ be a measure-preserving system.  Since $\frac{|\Delta(n)|}{|S(n)|} \rightarrow 1$, we need only show that for all $f \in L \phi (L)$,
\begin{equation*}
 \frac{1}{|S(n)|} \sum_{m \in \Delta(n)} f(T^mx)
\end{equation*}
converges almost everywhere.\\
We may assume $f \geq 0$.  Since $S$ is already universally $\frac{L^q}{\phi \left( L \right)}$ (and hence $L^q$) good, and since
\begin{equation*}
 \frac{1}{|S(n)|} \sum_{m \in \Delta(n)} f(T^mx) = \frac{1}{|S(n)|} \left\{ \sum_{m \in S(n)} f(T^mx) + \sum_{m \in \Delta(n) \setminus S} f(T^mx) \right\},
\end{equation*}
it will suffice if we show that for any $f \in L^q$
\begin{equation} \label{extragozero}
  \limsup \frac{1}{|S(n)|} \sum_{m \in \Delta(n) \setminus S} f(T^mx) = 0
\end{equation}

For arbitrary $n$, there exist $k$ and $u$ so that $n_k \leq n < n_{k+1}$. Since the  elements in $(\Delta \setminus S) \cap [n_k, n_{k+1})$ are contained in  $\Delta(2n_k) \setminus S$, we have 

\begin{equation*}
 \frac{1}{|S(n)|} \sum_{m \in \Delta(n) \setminus S} f(T^mx) \leq \frac{1}{|S(n_k)|} \sum_{m \in \Delta(2n_k) \setminus S} f(T^mx).
\end{equation*}

So we will have (\ref{extragozero}) if 
\begin{equation*}
 \int_X  \sum_{k=1} \left( \frac{1}{|S(n_k)|} \sum_{m \in \Delta(2n_k) \setminus S} f(T^mx) \right)^q  \, dm< \infty
\end{equation*}
for all $f \in L^q$.\\

Passing the integral inside the sum, and applying the triangle inequality,
\begin{align*}
 \int_X  &\sum_{k=1}^\infty \left( \frac{1}{|S(n_k)|} \sum_{m \in \Delta(2n_k) \setminus S} f(T^mx) \right)^q  \, dm \\
&= \sum_{k=1}^\infty \left\|  \frac{1}{|S(n_k)|} \sum_{m \in \Delta(2n_k) \setminus S} f(T^mx) \right\|_{L^q}^q\\
&\leq \left\| f \right\|_{L^q}^q \, \sum_{k=1}^\infty \left( \frac{1}{|S(n_k)|} \sum_{m \in \Delta(2n_k) \setminus S} 1 \right)^q.
\end{align*}
But
\begin{align*}
 \left\| f \right\|_{L^q}^q \, &\sum_{k=1}^\infty \left( \frac{1}{|S(n_k)|} \sum_{m \in \Delta(2n_k) \setminus S} 1 \right)^q \\
&=  \left\| f \right\|_{L^q}^q \, \sum_{u=1}^\infty \sum_{k \in A_u} \left( \frac{1}{|S(n_k)|} \sum_{m \in \Delta(2n_k) \setminus S} 1 \right)^q.
\end{align*}
Recalling inequality (\ref{allless}),
\begin{align*}
\left\| f \right\|_{L^q}^q \, &\sum_{u=1}^\infty \sum_{k \in A_u} \left( \frac{1}{|S(n_k)|} \sum_{m \in \Delta(2n_k) \setminus S} 1 \right)^q \\
&< \left\| f \right\|_{L^q}^q \, \sum_{u=1}^\infty \sum_{k \in A_u} \left( \frac{2 R(u) |S(n_k)|}{|S(n_k)|}  1 \right)^q\\
&= 2^q \left\| f \right\|_{L^q}^q \, \sum_{u=1}^\infty M(u)\left( R(u) \right)^q.
\end{align*}
Since $\phi^{-1}(x) >> x^\frac{1}{q}$, we have
\begin{align*}
 \sum_{u=1}^\infty M(u)\left( R(u) \right)^q &= \sum_{u=1}^\infty \left( \left\lfloor \frac{\left\{ \phi^{-1}(u^3) \right\}^q}{u^3} \right\rfloor +1 \right) \cdot \frac{u}{\left( \phi^{-1} (u^3) \right)^q}\\
&\leq \sum_{u=1}^\infty \frac{1}{u^2} + \sum_{u=1}^\infty \frac{u}{\left( \phi^{-1} (u^3) \right)^q}\\
&< \sum_{u=1}^\infty \frac{1}{u^2} + C \sum_{u=1}^\infty \frac{1}{u^2}
\end{align*}
for some constant $C$.
\end{proof}

\section{Proofs of Theorem \ref{big} and Lemma \ref{notasgood}}
\begin{proof}[Proof of Theorem \ref{big}]

We will consider only $\phi(x) \leq (\log x)^2$.  If $\phi(x) << \psi(x)$, then we have
$L \psi(L) \subset L \phi(L)$.  In order to construct a sequence that is good for $L \log^5 L$, for example, but bad for $L^1$, we need only construct an $L^1$-bad perturbation that is good for $L \log^2(L)$; this sequence will remain good for $L \log^5(L)$.

Let $g(u)= \log \phi^{-1}(u^{4})$. Since $\phi(x) \leq (\log x)^2$, we have that $g(u) \geq u^2$.  
The construction of our perturbation and proof that $\Delta$ is universally $\infty$-sweeping out in $L^1$ proceeds exactly as in the proof of Theorem \ref{prebig}, with
\begin{align*}
 M(u) &= \left\lfloor 2^{g(u)} \right\rfloor, \\ 
 R(u) &= \frac{u^\frac{1}{2}}{2^{g(u)}}, \mbox{ and} \\
 F(n) &= \left\{ \begin{array}{l}  2^{g(u)}  , \quad  \mbox{if } n= 0 \mbox{ mod } M(u) \\ 0 , \quad  \mbox{ otherwise. }  \end{array} \right.
\end{align*}

It remains to show that $\Delta$ is universally $L \phi(L)$-good.

Again we find that it suffices to show that
\begin{equation} \label{extragozero1}
  \limsup \frac{1}{|S(n)|} \sum_{m \in \Delta(n) \setminus S} f(T^mx) = 0
\end{equation}
for every $f \in L \phi(L)$.  Since for arbitrary $n$, there exist $k$ and $u$ so that $n_k \leq n < n_{k+1}$, we have
\begin{equation*}
 \frac{1}{|S(n)|} \sum_{m \in \Delta(n) \setminus S} f(T^mx) \leq \frac{1}{|S(n_k)|} \sum_{m \in \Delta(2n_k) \setminus S} f(T^mx).
\end{equation*}
So we need only show that 
\begin{equation*}
\int_X \left\{ \sum_{k=1}^\infty \left( \frac{1}{|S(n_k)|} \sum_{m \in \Delta(2n_k) \setminus S} f(T^mx) \right)^2 \right\}^\frac{1}{2} \, d\mu(x) < \infty.
\end{equation*}

This follows from Theorem \ref{Yanos} so long as 
\begin{equation*}
 \left\| \left\{ \sum_{k=1}^\infty \left( \frac{1}{|S(n_k)|} \sum_{m \in \Delta(2n_k) \setminus S} f(T^mx) \right)^{2} \right\}^\frac{1}{2} \right\|_{L^p} \leq  \phi \left( e^\frac{1}{p-1} \right)  \left\| f \right\|_{L^p} ,
\end{equation*}
for all $p$, $1< p \leq 2$.

Now,
\begin{align*}
 &\left\| \left\{ \sum_{k=1}^\infty \left( \frac{1}{|S(n_k)|} \sum_{m \in \Delta(2n_k) \setminus S} f(T^mx) \right)^{2} \right\}^\frac{1}{2}  \right\|_{L^p}\\ 
	&= \left\{ \int_X \left( \sum_{k=1}^\infty \left( \frac{1}{|S(n_k)|} \sum_{m \in \Delta(2n_k) \setminus S} f(T^mx) \right)^{2} \right)^\frac{p}{2} \, d\mu(x) \right\}^\frac{1}{p}. 
\end{align*}
Our first goal is to move the integral inside the first sum.  Noting that the $l^\frac{2}{p}$ norm is less than the $l^1$ norm,  
\begin{align}
&\left\{ \int_X  \left( \sum_{k=1}^\infty \left( \frac{1}{|S(n_k)|} \sum_{m \in \Delta(2n_k) \setminus S} f(T^mx) \right)^{2} \right)^\frac{p}{2} \, d\mu(x) \right\}^\frac{1}{p} \notag\\ 
&< \left\{ \int_X \sum_{k=1}^{\infty} \left( \frac{1}{|S(n_k)|} \sum_{m \in \Delta(2n_k) \setminus S} f(T^mx) \right)^p \, d\mu(x) \right\}^\frac{1}{p} \notag\\
&\leq \left\{  \sum_{k=1}^{\infty} \int_X \left( \frac{1}{|S(n_k)|} \sum_{m \in \Delta(2n_k) \setminus S} f(T^mx) \right)^p \, d\mu(x) \right\}^\frac{1}{p} \notag\\ 
&= \left\{  \sum_{k=1}^{\infty} \left\|  \frac{1}{|S(n_k)|} \sum_{m \in \Delta(2n_k) \setminus S} f(T^mx) \right\|_{L^p}^p \, \right\}^\frac{1}{p} \label{Lpin}
\end{align}
Applying the triangle inequality as in the proof of Theorem \ref{prebig}, (\ref{Lpin}) is less than
\begin{equation*}
 \left\| f \right\|_{L^p} \left\{  \sum_{k=1}^{\infty} \left( \frac{1}{|S(n_k)|} \sum_{m \in \Delta(2n_k) \setminus S} 1 \right)^p \, \right\}^\frac{1}{p}.
\end{equation*}
Breaking up the sum over $k$, we have
\begin{align*}
 &\left\| f \right\|_{L^p}  \left\{  \sum_{k=1}^{\infty} \left( \frac{1}{|S(n_k)|} \sum_{m \in \Delta(2n_k) \setminus S} 1 \right)^p \, \right\}^\frac{1}{p} \\
&= \left\| f \right\|_{L^p} \left\{ \sum_{u=1}^{\infty} \sum_{k \in A_u} \left( \frac{1}{|S(n_k)|} \sum_{m \in \Delta(2n_k) \setminus S} 1 \right)^p \, \right\}^\frac{1}{p} \\
&\leq \left\| f \right\|_{L^p} \left\{ \sum_{u=1}^{\infty} 2^{g(u)} \left( \frac{1}{|S(n_k)|} \left\{ \frac{u^\frac{1}{2}}{2^{g(u)}} \cdot 2|S(n_k)| \right\} \right)^p \, \right\}^\frac{1}{p} \\
&= 2 \left\| f \right\|_{L^p} \left\{ \sum_{u=1}^\infty \frac{u^\frac{p}{2}}{2^{g(u) \cdot (p-1)}} \right\}^\frac{1}{p}\\
&\leq 2 \left\| f \right\|_{L^p} \left\{ \sum_{u=1}^\infty \frac{u}{2^{g(u) \cdot (p-1)}} \right\}^\frac{1}{p}.
\end{align*}

We now wish to show that
\begin{equation*}
 \sum_{u=1}^\infty \frac{u}{2^{g(u) \cdot (p-1)}} \leq \phi\left( e^\frac{1}{p-1} \right).
\end{equation*}
It is in the course of providing an upper estimate for this sum that we make use of our requirement that  $\phi(x) \leq (\log x)^2$.  In the search for this upper estimate, we will consider two separate cases.  

First, suppose $g \sim u^n$ for some real number $n$.  Then there are there are nonzero constants $c$ and $C$ so that $cu^n < g(u) \leq Cu^n$. 
 Let 
\begin{align*}
 \alpha &= \frac{n}{n-1} \mbox{, and }\\
 N&= g^{-1}\left( \frac{1}{p-1} \right).
\end{align*}

Then 
\begin{align*}
 \sum_{u=1}^\infty \frac{u}{2^{g(u) \, (p-1)}} &= \sum_{u \leq N^\alpha} \frac{u}{2^{g(u) \, (p-1)}} + \sum_{u > N^\alpha} \frac{u}{2^{g(u) \, (p-1)}} \\
&\leq N^\alpha \sum_{u \leq N^\alpha} \frac{1}{2^{g(u) \, (p-1)}} + \sum_{u > N^\alpha} \frac{u}{2^{g(u) \, (p-1)}} \\
&\leq N^{2\alpha} + \sum_{u > N^\alpha} \frac{u}{2^{g(u) \, (p-1)}}.
\end{align*}

We claim that the second sum above is bounded by a constant.
If $u > N^\alpha$, then we have
\begin{align*}
 g(u)\, (p-1) &> cu^n (p-1) \\
&>c u N^{\alpha(n-1)} (p-1)\\
&=c u N^n (p-1)\\
&\geq \frac{c}{C} u g(N) (p-1)\\
&= \frac{c}{C} u.
\end{align*}

Thus there is a nonzero constant $K$, dependent on $g$ but independent of $p$, so that 
\begin{equation*}
 \sum_{u > N^\alpha} \frac{u}{2^{g(u) \, (p-1)}} \leq \sum_{r = 1}^\infty \frac{r}{2^{Kr}}.
\end{equation*}
This series is convergent regardless of what $K$ is, so the entire sum
\begin{equation*}
  \sum_{u=1}^\infty \frac{u}{2^{g(u) \, (p-1)}} \leq N^{2\alpha} + A,
\end{equation*}
where $A$ is some constant dependent only on $g$. Since $g(u) \geq u^2$, we have that $\alpha \leq 2$.  So, if $g \sim u^n$, we have
 \begin{align*}
  \sum_{u=1}^\infty \frac{u}{2^{g(u) \, (p-1)}} &\leq N^4 + A \\
&= \left( g^{-1} \left( \frac{1}{p-1} \right) \right)^4 + A \\
&= \phi \left( e^\frac{1}{p-1} \right) + A.
  \end{align*}

Now suppose $g >> u^n$ for all $n$.  Defining $\alpha$ and $N$ as before, we have
\begin{equation*}
 \sum_{u=1}^\infty \frac{u}{2^{g(u) \, (p-1)}} \leq N^{2\alpha} + \sum_{u > N^\alpha} \frac{u}{2^{u^n \, (p-1)}} \leq N^{2\alpha} + A,
\end{equation*}
where $A$ is again independent of $p$.  Letting $n \rightarrow \infty$, we have
\begin{align*}
 \sum_{u=1}^\infty \frac{u}{2^{g(u) \, (p-1)}}  &\leq N^2 + A \\
 &= \left( g^{-1} \left( \frac{1}{p-1} \right) \right)^2 + A \\
 &= \left( \phi \left( e^\frac{1}{p-1} \right) \right)^\frac{1}{2} +A.
 \end{align*}
\end{proof}

\begin{proof}[Proof of Lemma \ref{notasgood}]
 In this Lemma, we construct our perturbation by letting
\begin{align*}
 M(u) &= \left\lfloor 2^{g(u)} \psi \left( 2^{g(u)}\right)\right\rfloor, \\ 
 R(u) &= \frac{\left(\frac{u^k}{\psi \left( 2^{g(u)}\right)}\right)^\frac{1}{2}}{2^{g(u)}}, \mbox{ and} \\
 F(u) &= \left\{ \begin{array}{l}  2^{g(u)}  , \quad  \mbox{if } n= 0 \mbox{ mod } M(u) \\ 0 , \quad  \mbox{ otherwise. }  \end{array} \right.
\end{align*}
where $g(u)= \log \phi^{-1}\left( u^{k+1}\right)$.  
The proof that $\Delta$ is a perturbation proceeds as before.  We will need the requirement that $\frac{[\phi(u)]^\frac{k}{k+1}}{\psi(u)} \rightarrow \infty$ to show that $\Delta$ is universally $\infty$-sweeping out in $L \psi (L)$.

In order to show that the perturbation is bad for $L \psi(L)$, we will once again seek to apply Lemma \ref{badseq}.

With $\Psi(x) = |x| \, \psi(|x|)$, we have
\begin{align*}
 D\left( \Psi \left( F \right) \right) &=  \limsup_{N \rightarrow \infty} \frac{1}{2N+1} \sum_{n = -N}^N \Phi \left( F(n) \right)\\
&=  \limsup_{N \rightarrow \infty} \frac{1}{2N+1} \sum_{\substack{-N \leq n \leq N \\ n \equiv 0 \mbox{ mod } M(u) }} 2^{g(u)} \psi \left( 2^{g(u)}\right) \\ 
&=  \limsup_{N \rightarrow \infty}   \frac{\left| \left\{-N \leq n \leq N : \, n \equiv 0 \mbox{ mod } M(u) \right\} \right|}{2N+1} \cdot 2^{g(u)} \psi \left( 2^{g(u)}\right)  \\
&= \frac{1}{M(u)} \cdot 2^{g(u)} \psi \left( 2^{g(u)}\right) \\
&\leq 1.
\end{align*}

Once more, fix $u$ large enough so that 
\begin{equation*}
 |\Delta (2n_k)| \leq 4|S(n_k)|.
\end{equation*}
We may then estimate 
\begin{align*}
 \max_{k \in A_u} \frac{1}{|\Delta(2N_k)|} \sum_{m \in \Delta(2N_k)} F(n+m) &\geq  \frac{1}{4|S(N_k)|} \sum_{m \in E_k} F(n+m) \\
&= \frac{1}{4|S(N_k)|} \sum_{m \in E_k} 2^{g(u)}.
\end{align*}
As before, there must be at least $ R(u)|S(N_k)|$ elements in $E_k$. So
\begin{align*}
 \frac{1}{4|S(N_k)|} \sum_{m \in E_k} 2^{g(u)}
&\geq \frac{1}{4|S(N_k)|} \cdot R(u)|S(N_k)| \cdot 2^{g(u)}\\
&= \frac{1}{4} \cdot \left( \frac{u^k}{\psi \left( 2^{g(u)} \right)} \right)^\frac{1}{2}.
\end{align*} 
If $u^k >> \psi \left( 2^{g(u)} \right)$ we will then have that $\Delta$ is universally $\infty$-sweeping out in $L \psi (L)$ by Lemma \ref{badseq}.  

But $u^k >> \psi \left( 2^{g(u)} \right)$ whenever $\psi^{-1} \left( u^k \right) >> \phi^{-1} \left( u^{k+1} \right)$.  Since $\frac{[\phi(u)]^\frac{k}{k+1}}{\psi(u)} \rightarrow \infty$, we have that $ \phi^k >> \psi^{k+1}$.  So by Lemma \ref{badseq}, $\Delta$ is universally $\infty$-sweeping out in $L \psi (L)$.

As in Theorem \ref{big}, we will prove that our perturbation $\Delta$ remains good for $L \phi (L)$ by showing that 
\begin{equation} 
  \limsup \frac{1}{|S(n)|} \sum_{m \in \Delta(n) \setminus S} f(T^mx) = 0. \label{again}
\end{equation}
We will show that 
\begin{equation*}
 \left\| \left\{ \sum_{k=1}^\infty \left( \frac{1}{|S(n_k)|} \sum_{m \in \Delta(2n_k) \setminus S} f(T^mx) \right)^{2} \right\}^\frac{1}{2} \right\|_{L^p} \leq  \phi \left( e^\frac{1}{p-1} \right)  \left\| f \right\|_{L^p} ,
\end{equation*}
for all $p$, $1 < p \leq 2$, arriving at (\ref{again}) through the extrapolation theorem.
\newline

Proceeding as in the previous proof, we find that
\begin{align*}
 &\left\| \left\{ \sum_{k=1}^\infty \left( \frac{1}{|S(n_k)|} \sum_{m \in \Delta(2n_k) \setminus S} f(T^mx) \right)^{2} \right\}^\frac{1}{2} \right\|_{L^p} \\
&\leq \left\| f \right\|_{L^p} \left\{  \sum_{k=1}^{\infty} \left( \frac{1}{|S(n_k)|} \sum_{m \in \Delta(2n_k) \setminus S} 1 \right)^p \, \right\}^\frac{1}{p} \notag \\
&= \left\| f \right\|_{L^p} \left\{ \sum_{u=1}^{\infty} \sum_{k \in A_u} \left( \frac{1}{|S(n_k)|} \sum_{m \in \Delta(2n_k) \setminus S} 1 \right)^p \, \right\}^\frac{1}{p}\\
&< \left\| f \right\|_{L^p}  \sum_{u=1}^{\infty} M(u) \left( 
\frac{1}{S(n_k)} \cdot 2R(u)\,S(n_k) \right)^p.
\end{align*}

But, 
\begin{align*}
  \sum_{u=1}^{\infty} M(u) \left( \frac{1}{S(n_k)} \cdot 2R(u)\,S(n_k) \right)^p
&\leq 4 \sum_{u=1}^{\infty}  2^{g(u)} \psi \left( 2^{g(u)}\right) \cdot \frac{\left(\frac{u^k}{\psi \left( 2^{g(u)}\right)}\right)^\frac{p}{2}}{2^{g(u)\,p}}\\
&\leq 4 \sum_{u=1}^{\infty} \frac{u^k}{2^{g(u)\,(p-1)}}.
\end{align*}

Because $\phi(x) << \log^j (x)$ for all $j$, we have that $g(u) >> u^\frac{k+1}{j}$
Letting $n = \frac{1}{j}$ and defining $N$ and $\alpha$ as before, we have that 
\begin{align*}
 \sum_{u=1}^{\infty} \frac{u^k}{2^{g(u)\,(p-1)}} &= \sum_{u \leq N^\alpha} \frac{u^k}{2^{g(u) \, (p-1)}} + \sum_{u > N^\alpha} \frac{u^k}{2^{g(u) \, (p-1)}} \\
&\leq N^{(k+1)\alpha}  + \sum_{u > N^\alpha} \frac{u^k}{2^{u^{kn} \, (p-1)}} \\
&\leq N^{(k+1)\alpha} + A,
\end{align*}
where $A$ is again a constant independent of $p$. 

Letting $n \rightarrow \infty$, we have $\alpha \rightarrow 1$, and
\begin{align*}
 \sum_{u=1}^{\infty} \frac{u^k}{2^{g(u)\,(p-1)}} &\leq N^{(k+1)} + A \\
&= \phi \left( e^\frac{1}{p-1} \right) + A.
\end{align*}

.
\end{proof}

\section{Questions}
In Lemma \ref{notasgood}, the requirement that $\phi^k >> \psi^{(k+1)}$ is a product of our method of overestimating the sum.  Likewise, note that if $g(u) \sim u^n$ for some $n$- that is, if $\phi(x) \geq C\log^j(x)$ for some $j$- the method above requires that  $\frac{[\phi(u)]^\frac{k}{2(k+1)}}{\psi(u)} \rightarrow \infty$ to achieve the result. The question of whether one can construct a sequence that is good for a particular Orlicz space but bad for any larger Orlicz space remains.  

Since Theorem \ref{big} shows that we can construct a sequence that is good for a fixed Orlicz space, but bad for $L^1$, one might ask whether we can construct a sequence good for all Orlicz spaces but bad for $L^1$.  As it turns out, there is no such sequence; $L^1$ is the union of all Orlicz spaces properly contained in it (see \cite{RaoRen}).  Since there is a sequence that is universally $L \log \log (L)$-good but universally $\infty$-sweeping out for $L^1$, there is a sequence universally $L \log^s(L)$-good, for all $s$, but universally $\infty$-sweeping out in $L^1$. What other families of functions have this property?

Given a family of functions $\{ \phi_\alpha \}_{\alpha \in A}$, we may construct a sequence that is universally good for $L \phi_\alpha (L)$ but universally $\infty$-sweeping out so long as there is an unbounded function meeting the requirements of Theorem \ref{big} that grows more slowly than any $\phi_\alpha$. In this manner, we can construct a sequence universally $L \phi(L)$-good for all $\phi$ where $\phi$ is one of Hardy's logarithmico-exponential functions by letting our slower function be
\begin{equation*}
 f(x)= 1 \chi_{[0,1)} + \left(1 + \log x \right)\chi_{[1,2)} + \left(1 + \log(2) + \log \log x\right)\chi_{[1,4)} + ...
\end{equation*}
The question remains, however, for larger families, such as functions in the intersection of all maximal Hardy Fields.

\bibliographystyle{amsplain}

\bibliography{paperbib}

\providecommand{\bysame}{\leavevmode\hbox to3em{\hrulefill}\thinspace}
\providecommand{\MR}{\relax\ifhmode\unskip\space\fi MR }
\providecommand{\MRhref}[2]{%
  \href{http://www.ams.org/mathscinet-getitem?mr=#1}{#2}
}
\providecommand{\href}[2]{#2}
\begin{thebibliography}{10}

\bibitem{Bellow1984}
A.~Bellow and V.~Losert, \emph{On sequences of density zero in ergodic theory},
  Conference in modern analysis and probability ({N}ew {H}aven, {C}onn., 1982),
  Contemp. Math., vol.~26, Amer. Math. Soc., Providence, RI, 1984, pp.~49--60.
  \MR{MR737387 (86c:28034)}

\bibitem{Bellow1989Perturb}
Alexandra Bellow, \emph{Perturbation of a sequence}, Adv. Math. \textbf{78}
  (1989), no.~2, 131--139. \MR{MR1029097 (91f:28009)}

\bibitem{Boshernitzan1996}
Michael Boshernitzan and M{\'a}t{\'e} Wierdl, \emph{Ergodic theorems along
  sequences and {H}ardy fields}, Proc. Nat. Acad. Sci. U.S.A. \textbf{93}
  (1996), no.~16, 8205--8207. \MR{MR1401511 (97f:28041)}

\bibitem{Bourgain1989Arithmetic}
Jean Bourgain, \emph{Pointwise ergodic theorems for arithmetic sets}, Inst.
  Hautes \'Etudes Sci. Publ. Math. (1989), no.~69, 5--45, With an appendix by
  the author, Harry Furstenberg, Yitzhak Katznelson and Donald S. Ornstein.
  \MR{MR1019960 (90k:28030)}

\bibitem{Buczolich2008squares}
Z.~Buczolich and R.~D. Mauldin, \emph{Divergent square averages}, 2008,
  available at http://arxiv.org/abs/math/0504067.

\bibitem{LaVictoire2009}
P.~LaVictoire, \emph{Universally {$L\sp 1$}-bad arithmetic sequences}, 2009,
  Available at http://arxiv.org/abs/0905.3865.

\bibitem{RaoRen}
M.~M. Rao and Z.~D. Ren, \emph{Theory of {O}rlicz spaces}, Monographs and
  Textbooks in Pure and Applied Mathematics, vol. 146, Marcel Dekker Inc., New
  York, 1991. \MR{MR1113700 (92e:46059)}

\bibitem{Reinhold1994}
Karin Reinhold-Larsson, \emph{Discrepancy of behavior of perturbed sequences in
  {$L\sp p$} spaces}, Proc. Amer. Math. Soc. \textbf{120} (1994), no.~3,
  865--874. \MR{MR1169889 (94e:28008)}

\bibitem{Wierdl1988}
M{\'a}t{\'e} Wierdl, \emph{Pointwise ergodic theorem along the prime numbers},
  Israel J. Math. \textbf{64} (1988), no.~3, 315--336 (1989). \MR{MR995574
  (90f:11062)}

\bibitem{Wierdl1998Perturb}
\bysame, \emph{Perturbation of plane curves and sequences of integers},
  Illinois J. Math. \textbf{42} (1998), no.~1, 139--153. \MR{MR1492044
  (98m:28038)}

\bibitem{Yano1951}
Shigeki Yano, \emph{Notes on {F}ourier analysis. {XXIX}. {A}n extrapolation
  theorem}, J. Math. Soc. Japan \textbf{3} (1951), 296--305. \MR{MR0048619
  (14,41c)}

\end{thebibliography}
\end{document}